\documentclass[11pt]{amsart}

\usepackage[a4paper,hmargin=3.5cm,vmargin=4cm]{geometry}
\usepackage{amsfonts,amssymb,amscd,amstext}

\usepackage[utf8]{inputenc}
\usepackage{hyperref}

\usepackage{verbatim}

\usepackage{times}
\usepackage{enumerate}
\usepackage{titlesec}
\usepackage{mathrsfs}

\usepackage{fancyhdr}
\input xy
\xyoption{all}

\titleformat{\section}%[display]
{\filcenter\bfseries\large} {\thesection{.}}{0.2cm}{}%[$\vspace*{-1.0cm}$]
\titleformat{\subsection}[runin]
{\bfseries} {\thesubsection{.}}{0.15cm}{}[.]
\titleformat{\subsubsection}[runin]
{\em}{\thesubsubsection{.}}{0.15cm}{}[.]
\usepackage[up,bf]{caption}
\newtheorem{theorem}{Theorem}[section]
\newtheorem{proposition}[theorem]{Proposition}

\newtheorem{lemma}[theorem]{Lemma}
\newtheorem{corollary}[theorem]{Corollary}

\theoremstyle{definition}

\newtheorem{remark}[theorem]{Remark}

\newtheorem{problem}[theorem]{Problem}
\newtheorem{example}[theorem]{Example}

\numberwithin{equation}{section}
\numberwithin{figure}{section}

\newcommand\Hcal{\mathcal{H}}

\newcommand\Cscr{\mathscr{C}}

\newcommand\B{\mathbb{B}}
\newcommand\C{\mathbb{C}}

\newcommand\CP{\mathbb{CP}}

\newcommand\N{\mathbb{N}}

\newcommand\R{\mathbb{R}}

\newcommand\T{\mathbb{T}}
\newcommand\Z{\mathbb{Z}}

\renewcommand\c{\mathbb{C}}
\newcommand\cd{\overline{\mathbb D}}
\newcommand\cp{\mathbb{CP}}
\renewcommand\d{\mathbb D}

\newcommand\n{\mathbb{N}}
\renewcommand\r{\mathbb{R}}

\renewcommand\t{\mathbb{T}}
\newcommand\z{\mathbb{Z}}

\newcommand\igot{\mathfrak{i}}

\renewcommand\igot{\mathfrak{i}}

\renewcommand\imath{\igot}

\newcommand\wt{\widetilde}

\newcommand\dist{\mathrm{dist}}

\newcommand\length{\mathrm{length}}

\newcommand\Flux{\mathrm{Flux}}

\def\dist{\mathrm{dist}}

\def\length{\mathrm{length}}
\def\Flux{\mathrm{Flux}}

\usepackage{color}

\begin{document}

\fancyhead[LO]{The Calabi-Yau problem}
\fancyhead[RE]{A.\ Alarc\'on and F.\ Forstneri\v c} 
\fancyhead[RO,LE]{\thepage}

\thispagestyle{empty}

%% Title
%\vspace*{1cm}
\begin{center}
{\bf\LARGE The Calabi-Yau problem for Riemann surfaces with
finite genus and countably many ends}

\vspace*{0.5cm}

%% Authors
{\large\bf  Antonio Alarc{\'o}n \; and \; Franc Forstneri{\v c}} 
\end{center}

%% Addresses and finantial support
%\footnote[0]{\vspace*{-0.4cm}
%}
%% Abstract, keywords, and MSC

\vspace*{0.5cm}

\begin{quote}
{\small
\noindent {\bf Abstract}\hspace*{0.1cm}
In this paper, we show that if $R$ is a compact Riemann surface and 
$M=R\setminus\, \bigcup_i D_i$ is a domain in $R$ whose complement  is a union of 
countably many pairwise disjoint smoothly bounded closed discs $D_i$, then 
there is a complete conformal minimal immersion $X:M\to\R^3$, extending to a continuous 
map $X:\overline M\to\R^3$ such that $X(bM)=\bigcup_i X(bD_i)$ is a union of pairwise 
disjoint Jordan curves. In particular, $M$ is the complex structure of a complete bounded minimal 
surface in $\R^3$. This extends a recent result for bordered Riemann surfaces.

\vspace*{0.2cm}

\noindent{\bf Keywords}\hspace*{0.1cm} Riemann surface, minimal surface, 
Calabi-Yau problem
\vspace*{0.1cm}

\noindent{\bf MSC (2010):}\hspace*{0.1cm} 53A10, 53C42; 32B15, 32H02}
%
%  FF: these are good classification numbers
%
%  53A10: Minimal surfaces, surfaces with prescribed mean curvature
%  53C42: (1980-now) Immersions (minimal, prescribed curvature, tight, etc.)
%  32B15: (1973-now) Analytic subsets of affine space
%  32H05: Holomorphic mappings, embeddings,...
%
%  32E10: Stein manifolds
%  32E30: (1973-now) Holomorphic and polynomial approximation, Runge pairs, interpolation
%  32H02: (1991-now) Holomorphic mappings, (holomorphic) embeddings and related questions
%  32M17: Automorphism groups of Cn and affine manifolds 
%
\end{quote}

%%%%%%%%%%
%%%%%%%%%%
%%%%%%%%%%
%%%%%%%%%%
%%%%%%%%%%
%%%%%%%%%%

\vspace{1mm}

\section{Introduction} 
\label{sec:intro}

A classical problem in the theory of minimal surfaces in Euclidean spaces 
is the {\em conformal Calabi-Yau problem}, asking which open Riemann surfaces, $M$, admit a 
complete conformal minimal immersion $X:M\to\R^n$ $(n\ge 3)$ with bounded image. 
(Recall that a continuous map $X:M\to\R^n$ is said to be complete if the image of any divergent curve 
in $M$ has infinite Euclidean length. If $X$ is an immersion, this is equivalent to 
asking that the Riemannian metric $X^*(ds^2)$ on $M$, induced by the Euclidean metric 
$ds^2$ on $\R^n$ via $X$, is a complete metric.)
This problem originates in a conjecture of E.\ Calabi from 1965 that such immersions do not exist
(see Kobayashi and Eells \cite[p.\ 170]{Calabi1965Conjecture} and Chern \cite[p.\ 212]{Chern1966BAMS}). 
Groundbreaking counterexamples to Calabi's conjecture were given by Jorge and Xavier
\cite{JorgeXavier1980AM} in 1980 (a complete immersed minimal disc in $\R^3$ 
with a bounded coordinate function), Nadirashvili \cite{Nadirashvili1996IM} in 1996  
(a complete bounded immersed minimal disc in $\R^3$), and many others. 
In particular, there are examples in the literature of complete bounded minimal surfaces in $\R^3$ 
with any topological type (see Ferrer, Mart\'in, and Meeks \cite{FerrerMartinMeeks2012AM}). 
The related {\em asymptotic Calabi-Yau problem} (see S.-T.\ Yau \cite[p.\ 360]{Yau2000AMS})
asks about the asymptotic behaviour of such surfaces near their ends. We refer to the recent
papers \cite{AlarconDrinovecForstnericLopez2015PLMS,AlarconForstnericJAMS} for the history and  
literature on these problems.

The first main result of this paper is the following. 

%
%  MAIN THEOREM
%
\begin{theorem}\label{th:main}
Let $R$ be a compact Riemann surface. If $M=R\setminus \bigcup_{i=0}^\infty D_i$
is a domain in $R$ whose complement is a countable union of pairwise disjoint, 
smoothly bounded closed discs $D_i$ (diffeomorphic images of $\cd=\{z\in\C:|z|\le 1\}$), 
then $M$ is the complex structure of a complete bounded minimal surface in $\R^3$.

More precisely, for any $n\ge 3$ there exists a continuous map $X:\overline M\to\R^n$ such that 
$X:M\to \R^n$ is a complete conformal minimal immersion and $X(bM)=\bigcup_i X(bD_i)$ is the 
union of pairwise disjoint Jordan curves. If $n=4$ then $X:M\to\R^4$ can 
be chosen an immersion with simple double points, and
if $n\ge 5$ then $X:\overline M\to\R^n$ can be chosen an embedding.

The analogous result holds if $R$ is a nonorientable compact conformal surface.
\end{theorem}

The discs $D_i$ in the theorem are assumed to have boundaries of class $\Cscr^r$ for some
$r>1$, possibly noninteger. The Jordan curves $X(bD_i)\subset\R^n$ are everywhere 
nonrectifiable,  but we show that $X$ can be chosen such that they have Hausdorff dimension one. 

%
%  ABOUT OUR PREVIOUS RESULT
%
The analogue of Theorem \ref{th:main}  when $M$ is the complement of finitely many 
pairwise disjoint discs in a compact Riemann surface was obtained in 
\cite[Theorem 1.1]{AlarconDrinovecForstnericLopez2015PLMS}; this result also follows by a simplification 
of the proof of Theorem \ref{th:main}. Such $M$ is a {\em bordered Riemann surface} whose 
boundary, $bM$, consists of finitely many Jordan curves. The surfaces in Theorem \ref{th:main}
still have finite genus, but they may have countably many ends.

%
%  ABOUT THE PROOF
%
Theorem \ref{th:main} is proved in Sect.\ \ref{sec:proof} by an inductive application of 
\cite[Lemma 4.1]{AlarconDrinovecForstnericLopez2015PLMS}
(see Lemma \ref{lem:EnlargingDiameter}) which shows how to increase
the intrinsic diameter of a conformal minimal immersion $M\to\R^n$ from a compact bordered
Riemann surface by an arbitrarily big amount, while at the same time keeping the map 
uniformly close to the given one. Lemma \ref{lem:lowerbound} provides 
an estimate of the intrinsic radius of the image surface from below during the inductive process.

%
%  ADDITIONS TO THE THEOREM
%
The proof of Theorem \ref{th:main} gives several additions. In particular, the complete conformal
minimal immersion $X:M\to\R^n$ can be chosen to have vanishing flux. 
Alternatively, a minor modification of the proof enables us to prescribe the flux of $X$ 
on any given finite family of classes in the first homology group $H_1(M,\z)$; 
however, we do not know whether $X$ can be chosen with arbitrarily prescribed flux map.
On the other hand, if we do not insist on controlling the flux of $X$, then we can choose any conformal 
minimal immersion $X_0:\overline {R\setminus D_0}\to\R^n$ $(n\ge 3)$ and find for any given number 
$\epsilon>0$ a map $X$ as in the theorem which is uniformly $\epsilon$-close to $X_0$ on $\overline M$.

We wish to emphasize that the class of domains in Theorem \ref{th:main} contains the conformal
classes of all Riemann surfaces of finite genus with at most countably many ends, none of which 
are point ends. Indeed, the uniformization theorem of Z.-X.\ He and O.\ Schramm 
\cite[Theorem 0.2]{HeSchramm1993} says that every open Riemann surface, $M'$, 
with finite genus and at most countably many ends is conformally equivalent 
to a {\em circle domain} in a compact Riemann surface $R$, i.e., a domain of the form
\begin{equation}\label{eq:circle domain}
	M= R\setminus \bigcup_{i} D_i
\end{equation} 
whose complement is the union of at most countably many connected components $D_i$
each of which is either a closed geometric disc or a point. Here, 
a {\em geometric disc} in a Riemann surface $R$ is a topological disc whose lifts in the universal cover 
$\wt R$ of $R$ (which is the disc, the Euclidean plane, or the Riemann sphere) are round discs in $\wt M$.
The ends $D_i$ of $M$ which are points are called {\em point ends}, while the others are called
{\em disc ends}. An {\em annular end} is a disc end which does not contain any limit points of other ends.
A {\em puncture end}, or simply a {\em puncture}, is an end which is conformally isomorphic 
to the punctured disc; it corresponds to an isolated boundary point of a domain of the form \eqref{eq:circle domain}.
The type of an end is independent of a particular representation of a given open 
Riemann surface as a circle domain, and hence the above notions are well defined 
for open Riemann surfaces in this class. 
The He-Schramm  theorem includes as a special case open Riemann surfaces of finite
topological type (i.e., with finitely generated first homology group $H_1(M,\Z)$)
and says that every such is conformally equivalent to a domain in a compact Riemann surface
whose complement consists of finitely many closed geometric discs and points.
(This was known earlier, see e.g.\ the paper by E.\ L.\ Stout \cite{Stout1965TAMS}.) 

In light of these results, Theorem \ref{th:main} gives the following immediate corollary.
The second statement follows from the fact that a bounded harmonic function extends harmonically 
across an isolated point, and hence a bounded complete conformal minimal surface
does not have any punctures.

%
%
%    COMPLETE SOLUTION OF THE CONFORMAL CALABI-YAU PROBLEM FOR
%    SURFACES OF FINITE TYPE
%
%
\begin{corollary} \label{cor:conformalCY}
Every open Riemann surface of finite genus and at most countably ends,
none of which are point ends, is the conformal structure of a complete bounded 
immersed minimal surface in $\R^3$, and of a complete bounded embedded minimal surface in $\R^5$.

An open Riemann surface of finite topological type admits a bounded complete 
conformal minimal immersion into $\R^n$ for some (and hence for any) integer $n\ge 3$ 
if and only if it has no point ends. 
\end{corollary}

%
%   COLDING AND MINICOZZI: EMBEDDED COMPLETE MINIMAL SURFACES IN R3
%
On the other hand, Colding and Minicozzi proved \cite[Corollary 0.13]{ColdingMinicozzi2008AM} 
that a complete {\em embedded} minimal surface of finite topology in $\R^3$ is necessarily proper 
in $\R^3$, hence unbounded; this was extended to surfaces of finite genus and countably 
many ends by Meeks, P\'erez, and Ros \cite[Theorem 1.3]{MeeksPerezRos-CY}.
Hence, Corollary \ref{cor:conformalCY} exposes a major dichotomy
between the immersed and the embedded conformal Calabi-Yau problem in dimension $3$.

%
%   On boundary regularity of conformal isomorphisms
%
\begin{remark} 
We recall the following classical results on the boundary regularity of conformal maps.
These show that we are free to precompose conformal minimal immersions $\overline M\to \R^n$ from 
a compact bordered Riemann surface, $\overline M=M\cup bM$, 
by conformal isomorphisms $M'\to M$, provided that both 
bordered Riemann surfaces $M$ and $M'$ have boundaries of class $\Cscr^r$ for some $r>1$; 
this does not affect the boundary regularity of the maps.
Indeed, any conformal isomorphism $\phi:M\to M'$ between two such surfaces
extends to a homeomorphism $\phi:\overline M\to\overline {M'}$ of their closures
by the seminal theorem of Carath\'eodory \cite{Caratheodory1913} from 1913. 
Furthermore, if the boundaries $bM$ and $bM'$ are smooth (of class $\Cscr^\infty$), 
then the extension $\phi : \overline M\to\overline {M'}$ is a smooth diffeomorphism 
by Painlev\'e's theorem from 1887 \cite{Painleve1887,Painleve1891}. Improvements of these results
were made by many authors. In particular, it was shown by Warschawski \cite{Warschawski1935} in 1935 
that if $bM$ is of class $\Cscr^{k,\alpha}$ for some $k\in\Z_+=\{0,1,2,\ldots\}$ and $0<\alpha<1$, then 
$\phi$ is of the same class $\Cscr^{k,\alpha}$ on $\overline M$. See Goluzin \cite{Goluzin1969} 
for more information. For the corresponding boundary regularity results for minimal surfaces, 
see Nitsche \cite{Nitsche1969IM2,Nitsche1969IM1}.
\qed\end{remark}

%
%   OPEN PROBLEMS
%
The situation for Riemann surfaces more general than those in Theorem \ref{th:main} is not understood
yet, and we mention the following open problems in this direction. 

\begin{problem}\label{prob:problem1}
(A) Let $M$ be a domain of the form $M=R\setminus K$ in a compact Riemann surface $R$,
where $K$ is a nonempty compact subset of $R$.
Assume that $M$ admits a nonconstant bounded harmonic function $h:M\to (a,b)$ 
which does not extend to a bounded harmonic function in any bigger domain in $R$.
Does $M$ admit a bounded complete conformal minimal immersion into $\R^3$?

\noindent (B) Is there an example of a complete bounded minimal surface in $\R^3$ 
whose underlying complex structure is $\C\setminus K$, where $K$ is a Cantor set in $\C$?
\end{problem}

Recall (see \cite{JorgeXavier1980AM,AlarconFernandez2011DGA,AlarconFernandezLopez2012CMH,AlarconFernandezLopez2013CVPDE}) that every nonconstant bounded harmonic function $h:M\to (a,b)\subset \R$ 
on an open Riemann surface $M$ is a component function of a complete conformal minimal 
immersion $X=(X_1,X_2,h):M\to\R^3$ whose range is therefore contained in the slab 
$\{x=(x_1,x_2,x_3)\in\R^3: a<x_3<b\}$. Note also that the complement of a compact set $K\subset \C$ 
of positive capacity admits nonconstant bounded holomorphic (hence harmonic) functions.
This shows that the above questions are very natural.

A particular case of problem (A) concerns surfaces with point ends on which disc ends cluster. 
We wish to thank Antonio Ros for having asked whether anything could be said about the conformal Calabi-Yau problem
in this case (private communication on May 17, 2019). This seems a difficult problem, and the answer 
may depend on how the disc ends approach the set of point ends. In Section \ref{sec:proof}
we prove the following positive result under the assumption that the compact set of point 
ends is at infinite distance from the interior of $M$.

%
%   CY for nonisolated point ends
%
\begin{theorem}\label{th:main2}
Let $R$ be a compact Riemann surface, and let $M$ be a domain in $R$ of the form
\begin{equation}\label{eq:MDq}
	M = R\setminus (D\cup  E),  
\end{equation}
where $E$ is a compact set in $R$ and $D=\bigcup_{i=0}^\infty D_i$ is the union of a 
countable family of pairwise disjoint closed geometric discs $D_i\subset R\setminus E$.
Fix a point $p_0\in M$ and set $M_i=R\setminus  \bigcup_{j=0}^i D_j$ for every $i\in\N$. If  
\begin{equation}\label{eq:infinitedistance} 
	\lim_{i\to\infty} \dist_{M_i}(p_0,E)=+\infty,
\end{equation}
then there exists a continuous map $X:\overline M\to \R^3$ such that $X|_M:M\to\R^3$ is
a complete conformal minimal immersion and $X|_{bD}:bD=\bigcup_{i=0}^\infty bD_i \to\R^3$
is a topological embedding. In particular,  $M$ is the complex structure of a complete bounded 
minimal surface in $\R^3$. The analogous result holds if $R$ is nonorientable.
\end{theorem}

The distance $\dist_{M_i}(p_0,E)$ is measured with respect to a Riemannian metric $\mathrm d$ on 
the ambient surface $R$. In particular, the set $E$ may consist of point ends of $M$, and it may even be 
a Cantor set. Our proof fails for point ends which are at finite distance from an interior point, 
and it remains an open problem to decide what can happen in such case.

%
% AN EXAMPLE
%
We give an example of a domain in $\CP^1$ satisfying the requirements in Theorem \ref{th:main2} 
that is inspired by the labyrinth constructed by Jorge and Xavier in \cite{JorgeXavier1980AM}.

\begin{example}\label{ex:pointends}
Let $0<a<b<1$ be a pair of numbers and let $\lambda>0$. Choose finitely many numbers 
$a<s_0<s_1<\cdots< s_k<b$. For each $j=1,\ldots, k$ let
\[ 
	\delta_j=\frac{s_j-s_{j-1}}3>0
\] 
and 
\[ 
	K_j=\{z\in\c: s_{j-1}+\delta_j\le |z| \le  s_j-\delta_j, \ \  |\arg((-1)^jz)|\ge \delta_j\},
\] 
where $\arg(\cdot)$ is the principal branch of the argument with values in $(-\pi,\pi]$.
Up to a slight enlargement of each $K_j$, we can assume that they are smoothly bounded closed discs, 
still being pairwise disjoint. Let $\T$ denote the unit circle in $\C$.
Setting $K=\bigcup_{j=1}^k K_j$, it turns out that $\dist_{\d\setminus K}(a\t,b\t)>\lambda$ provided 
the integer $k\ge 1$ is chosen sufficiently large. Assume that this is so and 
denote the resulting set $K$ by $K_{a,b,\lambda}$. 

Now, choose a decreasing sequence $1>b_1>a_1>b_2>a_2>\cdots$ with $\lim_{i\to\infty} a_i=0$. 
Set $R=\cp^1$, $D_0=\cp^1\setminus \d$, and let denote by $D_1,D_2,\ldots$ the components of 
$\bigcup_{j=1}^\infty K_{a_j,b_j,1}$, ordered so that $|z_i|>|z_j|$ for all $z_i\in D_i$ 
and $z_j\in D_j$ for any pair of indices $i<j$. It is clear that the domain $M=R\setminus (D\cup E)$ 
with $E=\{0\}$ satisfies condition \eqref{eq:infinitedistance} for $M_i=R\setminus  \bigcup_{j=0}^i D_j$ and 
any point $p_0\in M$.
\qed\end{example}

The idea in the previous example can also be used to show the following.

%
%   MAKING DISTANCE TO E INFINITE
%
\begin{proposition}\label{prop:infinitedistance}
Let $E$ be a proper compact subset of a compact Riemann surface $R$.
Then there is a sequence of closed, smoothly bounded, pairwise disjoint discs $D_i\subset R\setminus E$
$(i\in\Z_+)$ such that $M=R\setminus \big( \bigcup_{i=0}^\infty D_i \cup  E\big)$ is a domain
satisfying condition \eqref{eq:infinitedistance}. Hence, the domain $M$ is the complex structure of a complete 
bounded minimal surface in $\R^3$.
\end{proposition}

\begin{proof}
Choose a point $p_0\in R\setminus E$ and a Morse exhaustion function $\rho:R\setminus E\to\R$
with $\rho(p_0)<0$. There are sequences $0<a_1<b_1<a_2<b_2<\cdots$ 
converging to $+\infty$ such that $\rho$ has no critical values in $[a_j,b_j]$ for every $j\in\N$.
It follows that the set $A_j=\{p\in M: a_j\le \rho(p)\le b_j\}$ is a union of finitely many 
pairwise disjoint annuli for each $j$. By placing sufficiently
many closed pairwise disjoint geometric discs in the interior of each connected component
of $A_j$, similarly to what has been done in the above example, we make the length
of every path crossing $A_j$ longer than $1$. (The length is measured with respect to 
any given Riemannian metric on $R$.) Doing this for every $j\in \N$ yields a
countable sequence of pairwise disjoint closed discs $D_i\subset R\setminus (E\cup \{p_0\})$ 
such that the length of any path from $p_0$ to $E$ which avoids all the discs $D_i$ is infinite
and \eqref{eq:infinitedistance} holds.
\end{proof}

%
%   GENERALIZATIONS TO (NULL) HOLOMORPHIC AND LEGENDRIAN CURVES
%
The techniques developed in 
\cite{AlarconForstneric2015MA,AlarconDrinovecForstnericLopez2015PLMS,AlarconForstnericLopez2019CM}
furnish an analogue of Lemma \ref{lem:EnlargingDiameter} for immersed holomorphic 
curves in $\C^n$ $(n\ge 2)$, null holomorphic curves in $\C^n$ for $n\ge 3$, 
and holomorphic Legendrian curves in $\C^{2n+1}$ for $n\ge 1$.
Recall that null curves are holomorphic immersions $Z=(Z_1,\ldots,Z_n):M\to\C^n$
from an open Riemann surface $M$ satisfying the nullity condition
$
	(dZ_1)^2+(dZ_2)^2+\cdots + (dZ_n)^2 = 0.
$
Every holomorphic curve in $\C^n$ is also a minimal surface 
by Wirtinger's theorem \cite{Wirtinger1936}, while the real and the imaginary part of a 
null holomorphic curve in $\C^n$ are conformal minimal surfaces in $\R^n$. 
By following the proof of Theorem \ref{th:main} and using the analogues 
of Lemma \ref{lem:EnlargingDiameter} for the appropriate classes of holomorphic curves, 
one obtains the following result.

%
%   COMPLETE (NULL) HOLOMORPHIC CURVES WITH JORDAN BOUNDARIES
%
\begin{theorem}\label{th:complex}
Let $M$ be an open Riemann surface as in Theorem \ref{th:main}. 
Then, for any $n\ge 2$ there exists a continuous map 
$Z:\overline M\to\C^n$ such that $Z:M\to \C^n$ is a complete holomorphic immersion and 
$Z(bM)=\bigcup_i Z(bD_i)$ is the union of pairwise disjoint Jordan curves $Z(bD_i)$.
If $n\ge 3$ then $Z$ can be chosen an embedding and such that $Z|_M:M\to \C^n$ is a 
complete null holomorphic embedding, or (if $n$ is odd) a complete holomorphic 
Legendrian embedding.  
\end{theorem}

The analogue of Theorem \ref{th:main2} also holds for these classes of maps, 
and the questions in Problem \ref{prob:problem1} make sense for holomorphic (null, Legendrian) curves. 

Theorem \ref{th:complex} contributes to the body of results concerning
Yang's problem \cite{Yang1977JDG} from 1977, asking about the existence 
and boundary behaviour of bounded complete complex submanifolds of complex Euclidean spaces. 
For recents developments on this subject, see the papers 
\cite{Alarcon2018,AlarconForstneric2019,AlarconGlobevnik2017,AlarconGlobevnikLopez2019Crelle,AlarconLopez2016JEMS,Globevnik2015AM} and the authors' survey \cite{AlarconForstnericJAMS}.

%%%%%%%%%%
%%%%%%%%%%
%%%%%%%%%%
%%%%%%%%%%   PRELIMINARIES
%%%%%%%%%%
%%%%%%%%%%

\section{Preliminaries} \label{sec:preliminaries}  

Given a compact smooth manifold $M$ with nonempty boundary $bM$ and 
an interior point $p_0\in \mathring M=M\setminus bM$, we denote by $\Cscr=\Cscr(M,p_0)$ 
the set of paths $\gamma:[0,1]\to M$ with $\gamma(0)=p_0$ and $\gamma(1)\in bM$. 
(The word {\em path} always stands for a continuous path. In fact, we shall mainly use piecewise
$\Cscr^1$ paths.) Given a continuous map $X:M\to\R^n$, we define
\[
	\dist_X(p_0,bM) = \inf \bigl\{\length(X\circ\gamma) : \gamma\in \Cscr\bigr\} \in [0,+\infty],
\]
where $\length(\lambda)$ denotes the Euclidean length of a path $\lambda:[0,1]\to\R^n$, i.e.,
the supremum of the sums $\sum_{i=1}^m |\lambda(t_i)-\lambda(t_{i-1})|$ 
over all subdivisions $0=t_0<t_1<\ldots < t_m=1$ of the interval $[0,1]$.
The same definition applies to a compact domain $M$ with countably many
boundary components. The number $\dist_X(p_0,bM)$
is called the {\em intrinsic diameter} of $M$ with respect to the map $X$ and the 
point $p_0$. A change of the base point changes the intrinsic diameter by a constant.

Assume now that $X:M\to\R^n$ is a smooth immersion and let $g=X^* ds^2$ be the
induced Riemannian metric on $M$. Given a piecewise $\Cscr^1$ path 
$\gamma:[0,1]\to M$, we have that
\[
	\length(X\circ\gamma)=\length_g(\gamma)=\int_0^1 \|\dot\gamma(t)\|_g \, dt.
\]
It is well known and easily seen that $\dist_X(p_0,bM)=\dist_g(p_0,bM)$ is the 
infimum of the lengths of piecewise $\Cscr^1$ paths in the family $\Cscr(M,p_0)$.
Indeed, we can replace any path $\gamma$ in $M$ by a piecewise smooth path
which is not longer than $\gamma$ by taking a suitably fine subdivision
$0=t_0<t_1<\ldots < t_m=1$ of $[0,1]$ and replacing each segment 
$C_i=\{\gamma(t):t_{i-1}\le t\le t_i\}$ by the geodesic arc
connecting the points $\gamma(t_{i-1})$ and $\gamma(t_i)$.

We shall consider two bigger classes 
$\Cscr_d=\Cscr_d(M,p_0) \subset \Cscr_{qd}=\Cscr_{qd}(M,p_0)$ of piecewise $\Cscr^1$ paths
in $M$ with the given initial point $p_0\in \mathring M$. The first one, 
$\Cscr_d(M,p_0)$, consists of {\em divergent paths} $\gamma:[0,1)\to M$ with $\gamma(0)=p_0$, 
i.e., such that $\gamma(t)$ leaves any compact subset of $\mathring M$ as $t$ approaches $1$.
(However, the limit of $\gamma(t)$ as $t\to 1$ need not exist.) The class
$\Cscr_{qd}(M,p_0)$ consists of paths $\gamma:[0,1)\to M$ such that $\gamma(0)=p_0$ 
and $\gamma$ has a cluster point on $bM$, i.e., there is a sequence 
\begin{equation}\label{eq:cluster}
	0<t_1<t_2<\cdots <1\ \ \text{with} \ \ \lim_{j\to\infty}t_j=1
	\ \ \text{and} \lim_{j\to\infty} \gamma(t_j)=p\in bM. 
\end{equation}
We call such path {\em quasidivergent}. 

The following lemma shows that we get the same intrinsic diameter by using 
paths in the bigger family $\Cscr_{qd}$, and hence also by using paths in $\Cscr_d$. 

%
%   QUASIDIVERGENT PATHS
%
\begin{lemma}\label{lem:quasidivergent}
Let $g$ be a Riemannian metric on a compact $\Cscr^1$ manifold $M$ with boundary $bM$,
and let $p_0\in \mathring M$. For every path $\gamma\in\Cscr_{qd}$ we have that 
$\length_g(\gamma)\ge \dist_g(p_0,bM)$.
\end{lemma}

\begin{proof} 
Fix $\epsilon>0$. Let $U\subset M$ be a neighbourhood of 
$bM$ such that every point $q\in U$ can be connected to a point $p\in bM$ by 
an arc of length less than $\epsilon$. Given $\gamma\in \Cscr_{qd}$, there is $t_0\in [0,1)$ 
such that $\gamma(t_0)\in U$. Choose an arc $C\subset U$ with $\length_g(C)<\epsilon$ 
connecting $\gamma(t_0)$ to a point $p\in bM$. 
Let $\lambda:[0,1]\to M$ be a path such that $\lambda(t)=\gamma(t)$
for $t\in [0,t_0]$, and $\lambda(t)$ for $t\in [t_0,1]$ is a parametrization of $C$
with $\lambda(1)=p$. Then, 
\[
	\dist_g(p_0,bM)\le \length_g(\lambda)<\length_g(\gamma)+\epsilon.
\]
Letting $\epsilon\to 0$ we obtain $\length_g(\gamma)\ge \dist_g(p_0,bM)$. 
\end{proof}

The following lemma will enable us to control the intrinsic radius of a conformal minimal immersion
from below in the proofs of Theorems \ref{th:main} and \ref{th:main2}.

%
%   LEMMA: LOWER BOUND ON THE BOUNDARY DISTANCE UNDER SMALL PERTURBATION
%
\begin{lemma}\label{lem:lowerbound}
Let $M$ be a compact connected $\Cscr^1$ manifold with boundary $bM\ne\varnothing$, and 
let $X: M\to \R^n$ be a $\Cscr^1$ immersion. Given a point $p_0\in \mathring M$
and a number $\eta>0$, there exists a number $\epsilon>0$ such that for every continuous map 
$Y:M\to\R^n$ with $\|X-Y\|_{\Cscr^0(M)} := \max\{|X(p)-Y(p)|:p\in M\} <\epsilon$ we have that
\[
	\inf \bigl\{\length(Y\circ\gamma) : \gamma\in \Cscr_{qd}(M,p_0)\bigr\} \ge \dist_X(p_0,bM)-\eta.
\]
\end{lemma}

\begin{proof}
This obviously holds if $Y$ is uniformly $\Cscr^1$-close to $X$ on $\mathring M$
since small $\Cscr^1$ perturbations only change lengths of curves by a small amount. 
Furthermore, any $\Cscr^1$ structure on a manifold is equivalent to a $\Cscr^\infty$ structure
by a theorem of H.\ Whitney \cite[Lemma 24]{Whitney1936}.
Hence, we may assume that $M$ is a compact domain with $\Cscr^1$ 
boundary in a smooth manifold $\wt M$ and $X$ is a smooth immersion $X:\wt M\to\R^n$,
where the latter statement uses an approximation theorem of Whitney \cite{Whitney1934TAMS}.

Let $N\to \wt M$ denote the normal bundle of the immersion $X:\wt M\to\R^n$, 
so $\dim N=n$. We identify $\wt M$ with the zero section of $N$. 
By the tubular neighbourhood theorem, $X$ extends to a smooth 
immersion $F:N\to \R^n$ which agrees with $X$ on the zero section $\wt M$ of $N$. Then, 
$g=F^*(ds^2)$ is a smooth Riemannian metric on $N$ whose restriction to $\wt M$ is the 
metric $X^*(ds^2)$, and the map $F:(N,g)\to (\R^n,ds^2)$ is a local isometry.
Let $\dist_g$ denote the distance function on $N$ induced by the Riemannian metric $g$.

We claim that there is a neighbourhood $U\subset N$ of $M$ such that
\begin{equation}\label{eq:est1}
	\dist_{g,U}(p_0,bM) >  \dist_{g,M}(p_0,bM) - \eta/2 = \dist_X(p_0,bM) - \eta/2,
\end{equation}
where $\dist_{g,U}(p_0,bM)$ is the distance from $p_0$ to $bM$ over all paths in $U$ and
$\eta>0$ is as in the lemma. Here is an elementary proof. After shrinking $N$ and $\wt M$
around $M$ if necessary, there is a smooth retraction $\rho:N\to \wt M$ such that  the kernel $\ker (d\rho_x)$ 
of its differential at any point $x\in \wt M$ is the $g$-orthogonal complement 
of $T_x \wt M$ in $T_x N$. Since  $d\rho_x$ equals the identity on $T_x\wt M$, it follows that
$d\rho_x$ has $g$-norm $1$. Hence, for any $r>1$ there is a neighbourhood $U\subset N$ of $M$ 
such that $d\rho_x:T_x N\to T_{\rho(x)}N$ has $g$-norm less than $r$ 
for every point $x\in U$. For every path $\gamma:[0,1]\to U$ we then have 
$\length_g(\rho\circ \gamma)\le r \,\cdotp \length_g(\gamma)$. Choosing
$\gamma$ to be a path in $U$ connecting $\gamma(0)=p_0$ to a point $\gamma(1)\in bM$, we obtain
\[
	 \dist_{g,M}(p_0,bM) \le \length_g(\rho\circ \gamma) \le r \,\cdotp \length_g(\gamma). 
\]
(The first inequality holds even if $\rho\circ\gamma$ is not contained in $M$
since it then crosses $bM$ at some time $t_0\in (0,1)$, and the length of this shorter path
is still $\ge \dist_{g,M}(p_0,bM)$.)  Taking the infimum over all such paths $\gamma$ gives 
\[
	\dist_{g,M}(p_0,bM)\le r \,\cdotp \dist_{g,U}(p_0,bM).
\]
If $r$ is chosen close enough to $1$ then \eqref{eq:est1} holds, thereby proving the claim.

Since $F:N\to F(N)\subset \R^n$ is a local isometry and $M$ is compact, there is a 
number $\epsilon_0>0$ such that for every point $p\in M$, the closed ball 
\begin{equation}\label{eq:gball}
	B_g(p,\epsilon_0):=\{q\in N:\dist_g(p,q)\le \epsilon_0\}
\end{equation}
is contained in $U$ and $F$ maps $B_g(p,\epsilon_0)$ isometrically onto the closed 
Euclidean ball $\overline \B(X(p),\epsilon_0)\subset\R^n$. By decreasing $\epsilon_0$ 
if necessary we may assume that $0<\epsilon_0<\eta/2$.

Given a continuous map $Y:M\to\R^n$ satisfying 
\[
	\max_{p\in M} |Y(p)-X(p)| < \epsilon \le \epsilon_0,
\] 
the above implies that there is a unique continuous map $\wt Y:M\to U$ such that 
\begin{equation}\label{eq:est2}
	Y=F\circ \wt Y\quad \text{and}\quad 
	\dist_g(p,\wt Y(p))=|X(p)-Y(p)|<\epsilon \ \ \text{for all}\ p\in M.
\end{equation}

Let $\gamma\in\Cscr_{cd}$ be a quasidivergent path in $M$ with $\gamma(0)=p_0$.
Fix a number $\epsilon$ with $0<\epsilon<\epsilon_0/2$. There is a 
boundary point $p\in bM$ and $t_0\in (0,1)$ such that 
\begin{equation}\label{eq:closetop}
	\dist_g(\gamma(t_0),p)<\epsilon.
\end{equation}
By \eqref{eq:est2} we have that $Y\circ \gamma = F\circ \wt \gamma$, where the path
$\wt \gamma = \wt Y\circ \gamma:[0,1]\to U$ satisfies 
\begin{equation}\label{eq:est3}
	\dist_g(\gamma(t),\wt\gamma (t)) <\epsilon \ \ \text{for all}\ t\in [0,1).
\end{equation}
Since $F:(N,g)\to (\R^n,ds^2)$ is a local isometry, we also have that 
\begin{equation}\label{eq:length}
	\length_g(\wt \gamma) = \length(F\circ\wt\gamma) = \length(Y\circ \gamma).
\end{equation}
By \eqref{eq:closetop}, \eqref{eq:est3}, and the triangle inequality, 
the point $\wt \gamma(t_0)=\wt Y(\gamma(t_0))$ satisfies 
\[
	\dist_g(\wt \gamma(t_0),p)\le 
	\dist_g(\wt \gamma(t_0),\gamma(t_0)) + \dist_g(\gamma(t_0),p) 
	< \epsilon + \epsilon \le \epsilon_0< \eta/2. 
\]
By adding to the path $\wt \gamma:[0,t_0]\to N$ an arc in the ball $B_g(p,\epsilon)$ \eqref{eq:gball} 
of $g$-length $<\eta/2$ connecting the point $\wt \gamma(t_0)$ to $p\in bM$, we obtain  
a path $\lambda:[0,1]\to U$ connecting $\lambda(0)=p_0$ to $\lambda(1)=p\in bM$ 
such that 
\[
	\length_g(\lambda) < \length_g(\wt \gamma) + \eta/2.
\]
We obviously have $\length_g(\lambda)\ge \dist_{g,U}(p_0,bM)$. 
Together with \eqref{eq:est1} we obtain 
\[
	\length_g(\wt \gamma) > \length_g(\lambda) -\eta/2 
	\ge \dist_{g,U}(p_0,bM) -\eta/2  >  \dist_{X}(p_0,bM) - \eta.
\]
In view of \eqref{eq:length} it follows that $\length(Y\circ \gamma)>\dist_{X}(p_0,bM) - \eta$.
\end{proof}

The proof of Lemma \ref{lem:lowerbound} also applies to the distance from an interior point
$p_0\in \mathring M$ to any given nonempty compact subset $E$ of $M$.
The following result to this effect will be used in the proof of Theorem \ref{th:main2}.

%
%   LEMMA: LOWER BOUND ON THE BOUNDARY DISTANCE, V2
%
\begin{lemma}\label{lem:lowerbound2}
Let $M$ be a compact connected $\Cscr^1$ manifold (either closed or with boundary), 
and let $X: M\to \R^n$ be a $\Cscr^1$ immersion. Given a nonempty compact set $E\subset M$,
a point $p_0\in M\setminus E$, and a number $\eta>0$, there is a number $\epsilon>0$ such that 
for every continuous map $Y: M\to\R^n$ with $\|X-Y\|_{\Cscr^0(M)}<\epsilon$ 
and for every path $\gamma:[0,1)\to M\setminus E$ with $\gamma(0)=p_0$
such that $\gamma(t)$ has a limit point in $E$ as $t\to 1$ we have 
$
	\length(Y\circ\gamma) \ge \dist_X(p_0,E)-\eta.
$
\end{lemma}

The following lemma (see \cite[Lemma 4.1]{AlarconDrinovecForstnericLopez2015PLMS})
is the main ingredient in the proof of Theorem \ref{th:main}. It enables one to make the intrinsic diameter 
of an immersed conformal minimal surface arbitrarily big by a $\Cscr^0$ small deformation.

%
%  Lemma: Enlarging Diameter
%
\begin{lemma}\label{lem:EnlargingDiameter}
Let $M$ be a compact bordered Riemann surface, and let $X:M\to\R^n$ $(n\ge 3)$ 
be a conformal minimal immersion of class $\Cscr^1(M)$. 
Given a point $p_0\in \mathring M=M\setminus bM$, 
an integer $d\in\Z_+$, and numbers $\epsilon>0$ (small) and $\mu>0$ (big), there is a 
continuous map $Y:M\to \R^n$ whose restriction to $\mathring M$ is a conformal minimal 
immersion such that the following conditions hold.
\begin{enumerate}[\rm (i)]
\item $|Y(p)-X(p)|<\epsilon$ for all $p\in M$.
\item $\dist_Y(p_0,bM)>\mu$.
\item $Y|_{bM}:bM\to \R^n$ is injective.
\item $\Flux_Y=\Flux_X$.
\end{enumerate}
\end{lemma}

Condition (iii) follows from a general position theorem; 
see \cite[Theorem 4.5]{AlarconDrinovecForstnericLopez2015PLMS}. The analogous result 
holds if $M$ is a nonorientable compact bordered conformal surface; 
see \cite[Section 6.3]{AlarconForstnericLopezMAMS} and in particular
Lemma 6.7 in the cited paper.

%%%%%%%%%%
%%%%%%%%%%
%%%%%%%%%%
%%%%%%%%%%   PROOF OF THEOREM 1.1
%%%%%%%%%%
%%%%%%%%%%

\section{Proof of Theorems \ref{th:main} and \ref{th:main2}} \label{sec:proof}

\noindent{\em Proof of Theorem \ref{th:main}.}
Assume that $R$ is a compact Riemann surface and $M$ is a domain in $R$ of the form
\begin{equation}\label{eq:M}
	M=R\setminus \bigcup_{i=0}^\infty D_i,
\end{equation}
where $\{D_i\}_{i\in\Z_+}$ is a countable family of closed, pairwise disjoint, smoothly boun\-ded
discs in $R$. We shall construct a continuous map $X: \overline M\to\r^n$ satisfying the conclusion of the theorem 
and such that the Jordan curves $X(bD_i)$, $i\in \Z_+$, have Hausdorff dimension one. 
Moreover, we shall ensure that the complete conformal minimal immersion $X:M\to\R^n$ has vanishing flux. 

For every $i=0,1,2,\ldots$ we let
\begin{equation}\label{eq:Mi}
	M_i=R\setminus\bigcup_{k=0}^i \mathring D_k.
\end{equation}
This is a compact bordered Riemann surface with boundary $bM_i=\bigcup_{k=0}^i bD_k$, and
\[
	M_0 \supset M_1\supset M_2\supset \cdots \supset \bigcap_{i=1}^\infty M_i = \overline M.
\]
By \cite[Theorem 4.5 (a)]{AlarconDrinovecForstnericLopez2015PLMS} there exists 
a conformal minimal immersion $X_0:M_0\to\R^n$ of class $\Cscr^1(M_0)$ with vanishing flux 
such that $X_0|_{bM_0}\colon bM_0\to\R^n$ is injective. 
Choose a Riemannian distance function ${\mathrm d}$ on $R$, a point $p_0\in M$, and 
a pair of numbers $\epsilon_0>0$ and $\tau_0\in\n=\{1,2,3,\ldots\}$.
An inductive application of Lemmas \ref{lem:lowerbound} and \ref{lem:EnlargingDiameter} 
furnishes a sequence of conformal minimal immersions $X_i:M_i\to\R^n$
of class $\Cscr^1(M_i)$, numbers $\epsilon_i>0$, and integers $\tau_i>i$ satisfying 
the following conditions for every $i\in\n$.
\begin{enumerate}
\item[\rm (a$_i$)] $\dist_{X_i}(p_0,bM_i) > i$.
\item[\rm (b$_i$)] $X_i:bM_i\to\R^n$ is injective.
\item[\rm (c$_i$)] $\sup_{p\in M_i} |X_i(p)-X_{i-1}(p)| < \epsilon_{i-1}$.
\item[\rm (d$_i$)] For every continuous map $Y:M_i\to\R^n$ with $\|Y-X_i\|_{\Cscr(M_i)} < 2\epsilon_i$ 
we have 
\[
	\inf \bigl\{\length(Y\circ\gamma) : \gamma\in \Cscr_{qd}(M_i,p_0)\bigr\} 
	>\dist_{X_i}(p_0,bM_i)-1 > i-1.
\]
\item[\rm (e$_i$)] We have $0<\epsilon_i<\frac12\min\left\{ \epsilon_{i-1} , \delta_{i},\tau_i^{-i}\right\}$, where
\[
	\delta_{i} := \frac1{i^2} 
	\inf\left\{ |X_{i}(p)-X_{i}(q)|: p,q\in bM_{i},\, {\rm d}(p,q)>\frac1{i}\right\} >0.
\]
\item[\rm (f$_i$)] We have $\tau_i>\tau_{i-1}$ and for each $k\in\{0,\ldots,i\}$ there is a set 
$A_{i,k}\subset X_i(bD_k)$ consisting of $\tau_i^{i+1}$ points such that 
\[
	\max\big\{\dist (p,A_{i,k})\colon p\in X_i(bD_k)\big\}<\frac1{\tau_i^i},
\]
where $\dist(p,A_{i,k})=\min\{|p-q|\colon q\in A_{i,k}\}$ is the Euclidean distance in $\R^n$.
\item[\rm (g$_i$)] $X_i$ has vanishing flux.
\end{enumerate}

Let us explain the induction step. Assume that for some $i\in \N$ we have maps $X_0,\ldots, X_{i-1}$
and numbers $\epsilon_0,\ldots,\epsilon_{i-1}$ and $\tau_0,\ldots,\tau_{i-1}$ satisfying these conditions for the respective
values of the index. (This holds for $i=1$ by using $X_0$, $\epsilon_0$, and $\tau_0$, and the 
above conditions are void except for (a$_0$), (b$_0$), and  (g$_0$); the second part of  (f$_0$) also 
holds true if we choose $\tau_0\in\n$ sufficiently large.)
Lemma \ref{lem:EnlargingDiameter} applied to $X_{i-1}|_{M_i}$ furnishes 
a conformal minimal immersion $X_i:M_i\to\R^n$ satisfying (a$_i$), (b$_i$), (c$_i$), and (g$_i$); 
note that $X_{i-1}|_{M_i}$ is flux vanishing since so is $X_{i-1}$ by (g$_{i-1}$).
Pick $\tau_i\in\n$ so large that condition (f$_i$) is satisfied; it suffices to choose 
\[
	\tau_i>\tau_{i-1}+\sum_{k=0}^i\length(X_i(b D_k)).
\]
Pick a number $\epsilon_i>0$ satisfying 
condition (e$_i$); such exists since $X_i|_{bM_i}$ is injective by (b$_i$). Finally, decreasing 
$\epsilon_i>0$ if necessary we may assume that condition (d$_i$) holds as well in view of Lemma 
\ref{lem:lowerbound}. The induction may proceed.

Conditions  (c$_i$) and (e$_i$) imply that the sequence $X_i$ converges uniformly on 
$\overline M$ \eqref{eq:M} to a continuous map $X=\lim_{i\to\infty}:\overline M\to\R^n$ 
whose restriction to $M$ is a conformal minimal immersion $X: M\to\R^n$, 
provided that each $\epsilon_i>0$ is chosen sufficiently small.
More precisely, for every $p\in \overline M$ we have that
\begin{equation}\label{eq:XXi}	
	|X(p)-X_{i}(p)|\le \sum_{k=i}^\infty |X_{k+1}(p)-X_{k}(p)|  
	< \sum_{k=i}^\infty  \epsilon_k < 2\epsilon_i.
\end{equation}
We can extend $X$ from $\overline M$ to a continuous map $X:M_i\to\R^n$ such that the above
inequality holds for all $p\in M_i$. 

Conditions (a$_i$) and $0<\epsilon_i<\epsilon_{i-1}/2$ (see (e$_i$)) ensure that $X:M\to\R^n$ is complete. 
Indeed, consider any divergent path $\gamma:[0,1) \to M$ with 
$\gamma(0)=p_0$. There is an increasing sequence $0<t_1<t_2<\cdots <1$ with $\lim_{j\to\infty}t_j=1$
such that $\lim_{j\to\infty} \gamma(t_j)=p\in bM$ (cf.\ \eqref{eq:cluster}).
Then, $p\in bD_{i_0}$ for some $i_0\in \Z_+$, and hence $p\in bM_i$ for all $i\ge i_0$.
It follows that $\gamma$ is a quasidivergent path in the bordered Riemann surface
$M_i$ for any $i\ge i_0$. (See Sect.\ \ref{sec:preliminaries} for this notion.) 
Conditions (a$_i$), (d$_i$), and \eqref{eq:XXi} imply for any $i\ge i_0$ that
\[
	\length(X(\gamma)) > \dist_{X_i}(p_0,bM_i)-1 > i-1. 
\]
Letting $i\to +\infty$ shows that $\length(X(\gamma))=+\infty$. 

Conditions (b$_i$), (c$_i$), and (e$_i$) imply that the limit map $X:\overline M\to\R^n$
is injective on $bM=\bigcup_{i\in\Z_+} bD_i$ (see 
\cite[proof of Theorem 1.1]{AlarconDrinovecForstnericLopez2015PLMS} for the details), 
whereas (g$_i$) ensures that $X$ has vanishing flux. 

Finally, in order to see that all Jordan curves $X(bD_k)$ $(k\in\Z_+)$ 
have Hausdorff dimension one, pick such a $k$. By \eqref{eq:XXi}, (e$_i$), and (f$_i$), 
we have for each $i> k$ that 
\begin{equation}\label{eq:dim}
	\max\{\dist (p,A_{i,k})\colon p\in X(bD_k)\}<\frac2{\tau_i^i}.
\end{equation}
Since  $A_{i,k}$ consists of precisely $\tau_i^{i+1}$ points and $\tau_i^{i+1}(2/\tau_i^i)^{1+1/i}=2^{1+1/i}\le 4$ 
for every integer $i>k$, \eqref{eq:dim} implies that the Hausdorff measure $\Hcal^1(X(bD_k))$ is finite, 
and hence the Hausdorff dimension of $X(bD_k)$ is at most one (cf.\ \cite[Lemma 2.2]{MartinNadirashvili2007ARMA} 
or \cite[Sect.\ 4.1]{Alarcon2010TAMS}; see \cite{Morgan2009} for an introduction to the Hausdorff measure). 
On the other hand, $X(bD_k)$ is homeomorphic to the circle $\mathbb{\t}=\{z\in\C: |z|=1\}$ and 
hence its Hausdorff dimension is at least one, so it is one.

Furthermore, by using the general position argument for minimal surfaces
at every step of the proof (see \cite[Theorem 4.1]{AlarconForstnericLopez2016MZ})
we can ensure that the limit map $X:M\to\R^n$ is an immersion with simple double
points if $n=4$, and is an embedding if $n\ge 5$. See 
\cite[proof of Theorem 1.1]{AlarconDrinovecForstnericLopez2015PLMS} for the details.

This completes the proof of Theorem \ref{th:main}. The same proof applies in the nonorientable
case if we replace Lemma \ref{lem:EnlargingDiameter} by \cite[Lemma 6.7]{AlarconForstnericLopezMAMS}.
\qed

%
%   PROOF OF THEOREM \label{th:pointends}
%

\medskip
\noindent{\em Proof of Theorem \ref{th:main2}.}
For every $i=0,1,2,\ldots$ let $M_i=R\setminus\bigcup_{k=0}^i \mathring D_k$ 
be the compact domain \eqref{eq:Mi} in $R$.
Choose a conformal minimal immersion $X_0:M_0\to\R^3$. Then, the given metric
on $R$ is comparable on $M_0$ to the metric $g_0=(X_0)^*(ds^2)$ induced by $X_0$, 
and hence condition \eqref{eq:infinitedistance} holds for the latter metric as well. 
We shall use the same argument at every step when changing the metric.

By \eqref{eq:infinitedistance}  there is $i_1\in\N$ such that $\dist_{M_{i_1},g_0}(p_0,E)>1$.
Choose a conformal minimal immersion $X_1: M_{i_1}\to\R^3$ which approximates $X_0$ 
uniformly on $M_{i_1}$ and satisfies the conditions
\[
	\dist_{M_{i_1},X_1}(p_0,bM_{i_1})>1 \quad \text{and}\quad \dist_{M_{i_1},X_1}(p_0,E)>1.
\]
The first condition is achieved by Lemma \ref{lem:EnlargingDiameter},
while the second holds by Lemma \ref{lem:lowerbound2} provided the approximation is close enough.
The metric $g_1=(X_1)^*(ds^2)$ is comparable on $M_{i_1}$ with $g_0$.
Hence, by \eqref{eq:infinitedistance} there is an integer $i_2>i_1$ such that $\dist_{M_{i_2},g_1}(p_0,E)>2$. 
The same argument as before gives a conformal minimal immersion 
$X_2:M_{i_2}\to\R^3$ approximating $X_1$ uniformly on $M_{i_2}$ and satisfying
\[
	\dist_{M_{i_2},X_2}(p_0,bD_{i_2})>2 \quad \text{and}\quad \dist_{M_{i_2},X_2}(p_0,E)>2.
\]
Continuing inductively we get sequences of integers $i_1<i_2<\cdots$
and conformal minimal immersions $X_k:M_{i_k}\to\R^3$ $(k\in\N)$ satisfying
\begin{equation}\label{eq:distanceestimates}
	\dist_{M_{i_k},X_k}(p_0,bM_{i_k})>k  \quad \text{and}\quad \dist_{M_{i_k},X_k}(p_0,E)>k.
\end{equation}
Assuming as we may that $X_{k}$ approximates $X_{k-1}$ sufficiently closely uniformly on $M_{i_k}$
for every $k\in\N$, we can ensure as in the proof of Theorem \ref{th:main} 
that the sequence $X_k$ converges uniformly on the compact set
$M'= \bigcap_{i=0}^\infty \overline M_i$ to a continuous limit map $X:M'\to \R^3$ 
whose restriction to the interior of $M'$ is a complete conformal minimal immersion,
and such that $X|_M:M\to\R^n$ satisfies the conclusion of the theorem.
(Here, $M=\mathring M' \setminus E$ is given by \eqref{eq:MDq}.)
In particular, the distance from any point of $M$ to the boundary $bM=\bigcup_{i}bC_i \cup bE$
in the metric $X^*(ds^2)$ is infinite by \eqref{eq:distanceestimates} and Lemma \ref{lem:lowerbound2}.
\qed

%%%%%%%%%%
%%%%%%%%%%
%%%%%%%%%%
%%%%%%%%%%   THANKS
%%%%%%%%%%
%%%%%%%%%%

\subsection*{Acknowledgements}
A.\ Alarc\'on is supported by the State Research Agency (SRA) and European Regional Development Fund (ERDF) via the grant no. MTM2017-89677-P, MICINN, Spain.
F.\ Forstneri\v c is supported  by the research program P1-0291 and the research grant 
J1-9104 from ARRS, Republic of Slovenia. The authors wish to thank anonymous referees
for their remarks and useful suggestions.

%%%%%%%%%%
%%%%%%%%%%
%%%%%%%%%%
%%%%%%%%%%   THE BIBLIOGRAPHY
%%%%%%%%%%
%%%%%%%%%%

%{\bibliographystyle{abbrv} \bibliography{references}}

%%%%%%%%%%
%%%%%%%%%%
%%%%%%%%%%
%%%%%%%%%%   AFFILIATIONS
%%%%%%%%%%
%%%%%%%%%%

%\newpage 

\vspace*{0.5cm}

\noindent Antonio Alarc\'{o}n

\noindent Departamento de Geometr\'{\i}a y Topolog\'{\i}a e Instituto de Matem\'aticas (IEMath-GR), Universidad de Granada, Campus de Fuentenueva s/n, E--18071 Granada, Spain

\noindent  e-mail: {\tt alarcon@ugr.es}

\vspace*{0.5cm}
\noindent Franc Forstneri\v c

\noindent Faculty of Mathematics and Physics, University of Ljubljana, Jadranska 19, SI--1000 Ljubljana, Slovenia

\noindent 
Institute of Mathematics, Physics and Mechanics, Jadranska 19, SI--1000 Ljubljana, Slovenia.

\noindent e-mail: {\tt franc.forstneric@fmf.uni-lj.si}

\end{document}